\documentclass[a4paper,12pt]{article}

\usepackage{amsmath}
\usepackage{amsfonts}
\usepackage{amssymb}
\usepackage{amstext    }
\usepackage{amsthm    }

\usepackage{mathrsfs}

\usepackage{hyperref}

\usepackage{color}

\usepackage[T1]{fontenc}
\usepackage{graphicx}

\usepackage{epsf}
\usepackage{psfrag}
\usepackage{subfigure}

\usepackage{fancyhdr}

\newtheorem{theorem}{Theorem}[section]
\newtheorem{lemma}[theorem]{Lemma}

\newtheorem{proposition}[theorem]{Proposition}
\newtheorem{remark}[theorem]{Remark}
\newtheorem{corollary}[theorem]{Corollary}
\newtheorem{conjecture}[theorem]{Conjecture}
\newcommand{\cali}[1]{\mathscr{#1}}

\numberwithin{equation}{section}

\newcommand{\dist}{{\rm dist}}

\title{Equidistribution speed towards the Green current for endomorphisms of $\mathbb P^k$}
\author{Johan Taflin}

\begin{document}

\maketitle

\begin{abstract}
Let $f$ be a non-invertible holomorphic endomorphism of $\mathbb P^k.$ For a hypersurface $H$ of $\mathbb P^k,$ generic in the Zariski sense, we give an explicit speed of convergence of $f^{-n}(H)$ towards the dynamical Green $(1,1)$-current of $f.$
\end{abstract}
\section{Introduction}\label{sec intro}
Let $f$ be a holomorphic endomorphism of algebraic degree $d\geq 2$ on the complex projective space $\mathbb P^k.$ The iterates $f^n=f\circ\cdots\circ f$ define a dynamical system on $\mathbb P^k.$ It is well-know that, if $\omega$ denotes the normalized Fubini-Study form on $\mathbb P^k$ then, the sequence $d^{-n}(f^n)^*(\omega)$ converges to a positive closed current $T$ of bidegree $(1,1)$ called the Green current of $f$ (see e.g. \cite{ds-lec}). It is a totally invariant current, whose support is the Julia set of $f$ and that exhibits interesting dynamical properties. In particular, for a generic hypersurface $H$ of degree $s,$ the sequence $d^{-n}(f^n)^*[H]$ converges to $sT$ \cite{ds-equi}. Here, $[H]$ denotes the current of integration on $H$ and the convergence is in the sense of currents. In fact, if we denote by $T^p$ the self-intersection $T\wedge\cdots\wedge T,$ Dinh and Sibony proposed the following conjecture on equidistribution.
\begin{conjecture}\label{conj equi}
Let $f$ be a holomorphic endomorphism of $\mathbb P^k$ of algebraic degree $d\geq 2$ and $T$ its Green current. If $H$ is an analytic set of pure codimension $p$ and of degree $s$ which is generic in the Zariski sense, then the sequence $d^{-pn}(f^n)^*[H]$ converges to $sT^p$ exponentially fast.
\end{conjecture}
The aim of the paper is to prove the conjecture for $p=1.$ It is a direct consequence of the following more precise result on currents. Indeed, we only have to apply the theorem to $S:=s^{-1}[H]$ for hypersurfaces $H$ which does not contain any element of $\cali A_\lambda.$
\begin{theorem}\label{th main}
 Let $f,$ $T$ be as above and let $1<\lambda<d.$ There exists a finite family $\cali A_\lambda$ of periodic irreducible analytic sets such that if $S$ is a positive closed $(1,1)$-current of mass $1,$ whose dynamical potential $u$ verifies $\|u\|_{L^1(X)}\leq C$ for all $X$ in $\cali A_\lambda,$ then the sequence $S_n:=d^{-n}(f^n)^*(S)$ converge exponentially fast to $T.$ More precisely, for every $0<\beta\leq 2$ and $\phi\in\cali C^\beta(\mathbb P^k)$ we get
\begin{equation}\label{eq main}
|\langle S_n-T,\phi\rangle|\leq A\|\phi\|_{\cali C^\beta}\left(\frac{\lambda}{d}\right)^{n\beta/2},
\end{equation}
where $A>0$ depends on the constants $C$ and $\beta$ but is independent of $S,$ $\phi$ and $n.$
\end{theorem}
Here, the space $L^1(X)$ is with respect to the volume form $\omega^{\dim(X)}$ on $X$ and $\cali C^\beta(\mathbb P^k)$ denotes the space of $(k-1,k-1)$-forms whose coefficients are of class $\cali C^\beta,$ equipped with the norm induced by a fixed atlas. \textit{The dynamical potential} of $S$ is the unique quasi-plurisubharmonic function $u$ such that $S=dd^cu+T$ and $\max_{\mathbb P^k}u=0.$ Note that $\cali A_\lambda$ will be explicitly constructed. Theorem \ref{th main} still holds if we replace $\cali A_\lambda$ by an analytic subset, e.g. a finite set, which intersects all components of $\cali A_\lambda.$

Equidistribution problem without speed was considered in dimension $1$ by Brolin \cite{brolin} for polynomials and by Lyubich \cite{lyubich-entropy} and Freire-Lopes-Ma\~né \cite{flm} for rational maps. They proved that for every point $a$ in $\mathbb P^1,$ with maybe two exceptions, the preimages of $a$ by $f^n$ converge towards the equilibrium measure, which is the counterpart of the Green current in dimension $1.$

In higher dimension, for $p=k,$ simple convergence in Conjecture \ref{conj equi} was established by Forn\ae ss-Sibony \cite{fs-cdhd1}, Briend-Duval \cite{briend-duval}. Recently in \cite{ds-speed}, Dinh and Sibony give exponential speed of convergence, which completes Conjecture \ref{conj equi} for $p=k.$ The equidistribution of hypersurfaces was proved by Forn\ae ss and Sibony for generic maps \cite{fs-cdhd2} and by Favre and Jonsson in dimension $2$ \cite{fj-brolin}. The convergence for general endomorphisms and Zariski generic hypersurfaces was obtained by Dinh and Sibony in \cite{ds-equi}. These papers state convergence but without speed. In other codimensions, the problem is much more delicate. However, the conjecture was solved for generic maps in \cite{ds-superpot}, using the theory of super-potentials.

We partially follow the strategy developed in \cite{fs-cdhd2}, \cite{fj-brolin} and \cite{ds-equi}, which is based on pluripontential theory together with volume estimates, i.e. a lower bound to the contraction of volume by $f.$ These estimates are available outside some exceptional sets which are treated using hypothesis on the map $f$ or on the current $S.$

The exceptional set $\cali A_\lambda$ will be defined in Section \ref{sec multi}. It is in general a union of periodic analytic sets possibly singular. In our proof of Theorem \ref{th main}, it is necessary to obtain the convergence of the trace of $S_n$ to these analytic sets. So, we have to prove an analog of Theorem \ref{th main} where $\mathbb P^k$ is replaced with an invariant analytic set. The geometry of the analytic set near singularities is the source of important technical difficulties. We will collect in Section \ref{sec loja} and Section \ref{sec vol} several versions of Lojasiewicz's inequality which will allow us to work with singular analytic sets and also to obtain good estimates on the size of a ball under the action of $f^n.$ Such estimates are crucial in order to obtain the convergence outside exceptional sets.

Theorem \ref{th main} can be reformulated as an $L^1$ estimate of the dynamical potential $u_n$ of $S_n$ (see Theorem \ref{th la version u_n du main}).
The problem is equivalent to a size control of the sublevel set $K_n=\{u_n\leq-(\lambda/d)^n\}.$ Since $T$ is totally invariant, we get that $u_n=d^{-n}u\circ f^n$ and $f^n(K_n)=\{u\leq-\lambda^n\}.$ The above estimate on the size of ball can be applied provide that $K_n$ is not concentrated near the exceptional sets. The last property will be obtained using several generalizations of exponential Hörmander's estimate for plurisubharmonic functions that will be stated in Section \ref{sec ho}. A key point in our approach is that, by reducing the domain of integration, we obtain uniform exponential estimates for non-compact families of quasi-plurisubharmonic functions.

We close this introduction by setting some notations and conventions. The symbols $\lesssim$ and $\gtrsim$ mean inequalities up to constants which only depend on $f$ or on the ambient space. To desingularize an analytic subset of $\mathbb P^k,$ we always use a finite sequence of blow-ups of $\mathbb P^k.$ Unless otherwise specified, the distances that we consider are naturally induced by embedding or smooth metrics for compact manifolds. For $K>0$ and $0<\alpha\leq1,$ we say that a function $u: X\to\mathbb C$ is $(K,\alpha)$-Hölder continuous if for all $x$ and $y$ in $X,$ we have $|u(x)-u(y)|\leq K\dist(x,y)^\alpha.$ We denote by $\mathbb B$ the unit ball of $\mathbb C^k$ and for $r>0,$ by $\mathbb B_r$ the ball centered at the origin with radius $r.$ In $\mathbb P^k,$ we denote by $B(x,r)$ the ball of center $x$ and of radius $r.$ And, for $X\subset \mathbb P^k$ an analytic subset, we denote by $B_X(x,r)$ the connected component of $B(x,r)\cap  X$ which contains $x.$ We call it the ball of center $x$ and of radius $r$ in $X.$ It may have more than one irreducible component. Finally, for a subset $Z\subset X,$ we denote by $Z_{X,t}$ or simply $Z_t,$ the tubular $t$-neighborhood of $Z$ in $X,$ i.e. the union of $B_X(z,t)$ for all $z$ in $Z.$ A function on $X$ is call (strongly) holomorphic if it has locally a holomorphic extension to a neighborhood of the ambient space.

\section{Lojasiewicz's inequality and consequences}\label{sec loja}
One of the main technical difficulties of our approach is related to singularities of analytic sets that we will handle using blow-ups along smooth varieties. In this section, we study the behavior of metric properties under blow-ups. It will allow us to establish volume and exponential estimates onto singular analytic sets.

We will frequently use the following Lojasiewicz inequalities. We refer to \cite{bm-subana} for further details.
\begin{theorem}\label{th iné de Loja classique pour les fonction}
Let $U$ be an open subset of $\mathbb C^k$ and let $h,g$ be subanalytic functions in $U.$ If $h^{-1}(0)\subset g^{-1}(0),$ then for any compact subset $K$ of $U,$ there exists a constant $N\geq1$ such that, for all $z$ in $K,$ we have
$$|h(z)|\gtrsim |g(z)|^N.$$
\end{theorem}
In this paper, we only use the notion of subanalytic function in the following case. Let $U$ be an open subset of $\mathbb C^k$ and $A\subset U$ be an analytic set. Every compact of $U$ has a neighborhood $V\subset U$ such that the function $x\mapsto \dist(x,A)$ and analytic functions on $U$ are subanalytic on $V.$ Moreover, the composition or the sum of two such functions is still subanalytic on $V.$ In particular, we have the following property.
\begin{corollary}\label{cor iné de Loja classique pour ens}
Let $U$ be an open subset of $\mathbb C^k$ and let $A,$ $B$ be analytic subsets of $U.$ Then for any compact subset $K$ of $U,$ there exists a constant $N\geq1$ such that, for all $z$ in $K,$ we have
$$\dist(z,A)+\dist(z,B)\gtrsim\dist(z,A\cap B)^N.$$
\end{corollary}

We briefly recall the construction of blow-up that we will use later. If $U$ is an open subset of $\mathbb C^k$ which contains $0$, the blow-up $\widehat U$ of $U$ at $0$ is the submanifold of $U\times \mathbb P^{k-1}$ defined by the equations $z_iw_j=z_jw_i$ for $1\leq i,j\leq k$ where $(z_1,\ldots,z_k)$ are the coordinates of $\mathbb C^k$ and $[w_1:\cdots:w_k]$ are the homogeneous coordinates of $\mathbb P^{k-1}.$ The sets $w_i\neq 0$ define local charts on $\widehat U$ where the canonical projection $\pi: \widehat U\to U,$ if we set for simplicity $i=1$ and $w_1=1,$ is given by
$$\pi(z_1,w_2,\ldots,w_k)=(z_1,z_1w_2,\ldots,z_1w_k).$$
If $V\subset \mathbb C^p$ is an open subset, \textit{the blow-up} of $U\times V$ along $\{0\}\times V$ is defined by $\widehat U\times V.$ This is the local model of a blow-up.

Finally, if $X$ is a complex manifold, the blow-up $\widehat X$ of $X$ along a submanifold $Y$ is obtained by sticking copies of the above model and by using suitable atlas of $X.$ The natural projection $\pi:\widehat X\to X$ defines a biholomorphism between $\widehat X\setminus\widehat Y$ and $X\setminus Y$ where the set $\widehat Y:=\pi^{-1}(Y)$ is called the exceptional hypersurface. If $A$ is an analytic subset of $X$ not contained in $Y,$ the strict transform of $A$ is defined as the closure of $\pi^{-1}(A\setminus Y).$

We have the following elementary lemma.
\begin{lemma}\label{le distance eclat}
Let $\pi:\widehat U\times V\to U\times V$ be as above and $\widehat Y$ denote the exceptional hypersurface in $\widehat U\times V.$ Assume that $U\times V$ is bounded in $\mathbb C^k\times\mathbb C^p.$ Then for all $\widehat z, \widehat z'\in\widehat U\times V,$ we have
$$\dist(z,z')\gtrsim \dist(\widehat z,\widehat z')(\dist(\widehat z,\widehat Y)+\dist(\widehat z',\widehat Y)),$$
where $z=\pi(\widehat z)$ and $z'=\pi(\widehat z').$
\end{lemma}
\begin{proof}
Since $\pi$ leaves invariant the second coordinate, considering the maximum norm on $U\times V,$ the general setting is reduced to the case of blow-up of a point. Hence we can take $V=\{0\}.$ The lemma is obvious if $z$ or $z'$ is equal to $0.$ Since $\pi$ is a biholomorphism outside $\widehat Y,$ we can assume that $0<\|z\|,\|z'\|<1.$ Moreover, up to an isometry of $\mathbb C^k,$ we can assume that $\max|z_i|=|z_1|$ and $\max|z'_i|=|z'_1|.$ Indeed, we can send $z$ and $z'$ into the plane generated by the first two coordinates and then use the rotation group of this plane. Therefore, in the chart $w_1=1,$ we have $\widehat z=(z_1,w_2,\ldots,w_k),$ $\widehat z'=(z'_1, w'_2,\ldots,w_k')$ with $|w_i|,|w'_i|\leq 1.$

By triangle inequality, we have
$$|z_1w_i-z'_1w'_i|\geq |z_1w_i-z_1w'_i|-|z_1w'_i-z'_1w'_i|\geq|z_1||w_i-w'_i|-|z_1-z'_1|.$$
Hence, by symmetry in $z_1$ and $z'_1$ we get
$$2|z_1w_i-z'_1w'_i|\geq(|z_1|+|z'_1|)|w_i-w'_i|-2|z_1-z'_1|.$$
Therefore, there is a constant $a>0$ independent of $z$ and $z'$ such that
$$|z_1-z'_1|+\sum_{i=2}^k|z_1w_i-z'_1w'_i|\geq a(|z_1|+|z'_1|)(|z_1-z'_1|+\sum_{i=2}^k|w_i-w'_i|).$$
The left-hand side corresponds to $\dist(z,z').$ We also have that $\dist(\widehat z,\widehat z')\simeq|z_1-z'_1|+\sum_{i=2}^k|w_i-w'_i|$ and $\dist(\widehat z,\widehat Y)+\dist(\widehat z',\widehat Y)\simeq|z_1|+|z'_1|.$ The result follows.
\end{proof}

A similar result holds for analytic sets. More precisely, consider an irreducible analytic subset $X$ of $\mathbb P^k$ of dimension $l$ and a smooth variety $Y$ contained in $X.$ Let $\overline \pi:\widehat{\mathbb P^k}\to\mathbb P^k$ be the blow-up along $Y$ and $\pi$ the restriction of $\overline \pi$ to the strict transform $\widehat X$ of $X.$ Denote by $\overline Y$ and $\widehat Y$ the exceptional hypersurfaces in $\widehat{\mathbb P^k}$ and in $\widehat X$ respectively.
  \begin{lemma}\label{le distance eclat sur ensemble ana}
There exists $N\geq1$ such that for all $\widehat z$ and $\widehat z'$ in $\widehat X$
$$\dist(z,z')\gtrsim \dist(\widehat z,\widehat z')(\dist(\widehat z,\widehat Y)+\dist(\widehat z',\widehat Y))^N,$$
\end{lemma}
\begin{proof}
The previous lemma gives the inequality with $N=1$ if we substitute $\widehat Y$ by $\overline Y.$ By Corollary \ref{cor iné de Loja classique pour ens} applied to $A=\widehat X$ and $B=\overline Y,$ there exists $N\geq1$ such that $\dist(\widehat x,\overline Y)\gtrsim\dist(\widehat x,\widehat Y)^{N}$ for all $\widehat x$ in $\widehat X.$ The result follows.
\end{proof}
Here is the first estimate on contraction for blow-ups.
\begin{lemma}\label{le vol de ball par eclat}
There exists a constant $N\geq1$ such that for all $0<t\leq1/2,$ if $\dist(\widehat x,\widehat Y)>t$ and $r<t/2,$ then $\pi(B_{\widehat X}(\widehat x,r))$ contains $B_X(\pi(\widehat x), t^Nr).$ Moreover, if $N$ is large enough then the image by $\pi$ of a ball of radius $0<r\leq1/2$ contains a ball of radius $r^N$ in $X.$
\end{lemma}
\begin{proof}
Let $\widehat y$ be a point in $\widehat X$ such that $\dist(\widehat x,\widehat y)=r$ and set $x=\pi(\widehat x),$ $y=\pi(\widehat y).$ The assumption on $r$ gives that $\dist(\widehat y,\widehat Y)>t/2.$ Therefore, we deduce from Lemma \ref{le distance eclat sur ensemble ana} that
$$\dist(x,y)\gtrsim rt^N.$$
The first assertion follows since $t\leq1/2$ and $\pi$ is a biholomorphism outside $\widehat Y.$

For a general ball $B$ of radius $r$ in $\widehat X,$ we can reduce the ball in order to avoid $\widehat Y$ and then apply the first statement. More precisely, as $\dim(\widehat Y)\leq l-1$ there is a constant $c>0$ such that for all $\rho>0,$ $\widehat Y$ is cover by $c\rho^{-2(l-1)}$ balls of radius $\rho.$ On the other hand, by a theorem of Lelong \cite{lelong-book,de-book}, the volume of a ball of radius $\rho$ in $\widehat X$ varies between $c'^{-1}\rho^{2l}$ and $c'\rho^{2l}$ for some $c'>0.$  Hence, the volume of $\widehat Y_\rho$ is of order $\rho^2.$ Take $\rho=c''r^l$ with $c''>0$ small enough. By counting the volume, we see that $B$ is not contained in $\widehat Y_\rho.$ Therefore, $B\setminus \widehat Y$ contains a ball of radius $\rho/3.$ We obtain the result using the first assertion.
\end{proof}
In the same spirit, we have the following lemma.
\begin{lemma}\label{le estime voisinage de transforme totale}
Let $\widehat Z$ be a compact manifold, $Z$ be a irreducible analytic subset of $\mathbb P^k$ and $\pi:\widehat Z\to Z$ be a surjective holomorphic map. Let $A$ be an irreducible analytic subset of $Z$ and define $\widehat A:=\pi^{-1}(A).$ There exists $N\geq1$ such that $A_{t^N}$ is included in $\pi(\widehat A_t)$ for all $t>0$ small enough. Moreover, if $\widehat A$ is the union on two analytic sets $\widehat A_1,$ $\widehat A_2$ such that $A_2:=\pi(\widehat A_2)$ is strictly contained in $A$ then $\pi(\widehat A_{1,t})$ contains $A_{t^N}\setminus A_{2,t^{1/2}}.$
\end{lemma}
\begin{proof}
Since $\dist(\widehat z,\widehat A)=0$ if and only if $\dist(\pi(\widehat z),A)=0,$ we can apply Theorem \ref{th iné de Loja classique pour les fonction} to these functions which implies the existence of $N\geq1$ such that
$$\dist(\widehat z,\widehat A)^{N}\lesssim\dist(\pi(\widehat z),A).$$
Therefore, since $\pi$ is surjective, $\pi(\widehat A_t)$ contains $A_{t^{N}}$ for $t>0$ small enough.

Now, let $\widehat A=\widehat A_1\cup\widehat A_2$ be as above. Since $\pi$ is holomorphic, there exists $c>0$ such that $\pi(\widehat A_{2,t})\subset A_{2,ct}.$ Therefore, $\pi(\widehat A_{1,t})$ contains $A_{t^N}\setminus A_{2,t^{1/2}}$ for $t>0$ sufficiently small.
\end{proof}

In the sequel, we will constantly use desingularization of analytic sets. The following lemma allows us to conserve integral estimates.

\begin{lemma}\label{le remonté d'estimée par surjection}
Let $Z$ and $\widehat Z$ be irreducible analytic subsets of Kähler manifolds. Let $\pi:\widehat Z\to Z$ be a surjective proper holomorphic map. Then, for every compact $\widehat L$ of $\widehat Z$ there exists $q\geq1$ such that if $v$ is in $L^q_{loc}(Z)$ then $\widehat v:=v\circ \pi$ is in $L^1(\widehat L).$ Moreover, there exists $c>0,$ depending on $\widehat L,$ such that
$$\|\widehat v\|_{L^{1}(\widehat L)}\leq c\|v\|_{L^{q}(\pi(\widehat L))}.$$
\end{lemma}
\begin{proof}
Using a desingularization, we can assume that $\widehat Z$ is a smooth Kähler manifold with a Kähler form $\widehat \omega.$ Denote by $\omega,$ $n,$ $m$ a Kähler form on $Z$ and the dimensions of $Z$ and $\widehat Z$ respectively. Generic fibers of $\pi$ are compact of dimension $m-n$ and form a continuous family. It follows that the integral of $\widehat \omega^{m-n}$ on that fibers is a constant.

Consider $\widehat \lambda=\pi^*(\omega^n)\wedge\widehat \omega^{m-n}$ on $\widehat Z.$ The last observation implies that $\pi_*(\widehat \lambda)=\omega^n$ up to a constant. Therefore, if $v$ is in $L_{loc}^q(Z)$ then $\widehat v$ is in $L_{loc}^q(\widehat Z,\widehat \lambda).$ Moreover, we can write $\widehat \lambda=h\widehat \omega^m$ where $h$ is a positive function. If there exists $\tau>0$ such that $h^{-\tau}$ is integrable on $\widehat L$ with respect to $\widehat \omega^m,$ we obtain for $p=1+\tau$ and $q$ its conjugate that
\begin{align*}
 \int_{\widehat L} |\widehat v|\widehat \omega^m=\int_{\widehat L} |\widehat v|h^{-1}\widehat \lambda&\leq \left(\int_{\widehat L}|\widehat v|^q\widehat \lambda\right)^{1/q}\left(\int_{\widehat L} h^{-p}\widehat \lambda\right)^{1/p}\\
&\lesssim \left(\int_{\pi(\widehat L)}|v|^q\omega^n\right)^{1/q}\left(\int_{\widehat L} h^{-\tau}\widehat\omega^m\right)^{1/p}\\
&\lesssim \|v\|_{L^q(\pi(\widehat L))}.
\end{align*}
It remains to show the existence of $\tau.$ The set $\{h=0\}$ is contained in the complex analytic set $A$ where the rank of $\pi$ is not maximal. More precisely, if $\pi$ has maximal rank at $z,$ then we can linearize $\pi$ in a neighborhood of $z.$ Therefore, $\widehat \lambda$ and $\widehat \omega^m$ are comparable in that neighborhood.

Since $\widehat L$ is compact, the problem is local. Let $z_0$ in $\widehat L.$ We can find a small chart $U$ at $z_0$ and a holomorphic function $\phi$ on $U$ such $A\cap U$ is contained in $\{\phi=0\}.$ We can also replace $\widehat \omega$ and $\omega$ by standard Euclidean forms on $U$ and on a neighborhood of $\pi(U).$ So, we can assume that $h$ is analytic. By Lojasiewicz inequality, for every compact $K$ of $U,$ there exists $N\geq1$ such that $h(z)\gtrsim|\phi(z)|^N$ for all $z$ in $K.$ On the other hand, exponential estimate (cf. \cite{ho-book} and Section \ref{sec ho}) applied to the plurisubharmonic function $\log|\phi|$ says that $\phi^{-\alpha}$ is in $L^1(K,\widehat \omega^m)$ for some $\alpha>0.$ Therefore, $h^{-\alpha/N}$ belongs to $L^1(K,\widehat\omega^m).$ We obtain the desired property near $z_0$ by taking $\tau\leq\alpha/N.$
\end{proof}
Finally, the following results establish a relation between the regularity of functions on an analytic set and that of their lifts to a desingularization.
\begin{proposition}\label{pro controle des coeff holder}
Let $Z,$ $\widehat Z,$ $\pi$ be as in Lemma \ref{le remonté d'estimée par surjection} and $v$ be a function on $Z.$ Assume that the lift of $v$ to $\widehat Z$ is $(K,\alpha)$-Hölder continuous. Then, for every compact $L$ of $Z,$ there exist constants $0<\alpha'\leq 1$ and $a>0,$ independent of $v,$ such that $v$ is $(aK,\alpha')$-Hölder continuous on $L.$
\end{proposition}
\begin{proof}
Let $\Delta$ be the diagonal of $Z\times Z.$ We still denote by $\pi$ the map induced on the product $\widehat Z\times\widehat Z$ and we set $\widehat \Delta=\pi^{-1}(\Delta).$ 
As in Lemma \ref{le estime voisinage de transforme totale}, by Lojasiewicz inequality, we have
$$\dist(\widehat a,\widehat\Delta)^M\lesssim\dist(\pi(\widehat a),\Delta),$$
if $\pi(\widehat a)\in L\times L.$ Therefore, if we set $\widehat a=(\widehat x,\widehat x')$ and $\pi(\widehat a)=(x,x'),$ then we can rewrite this inequality as
\begin{equation}\label{eq lemme coeff holder}
(\dist(\widehat x,\widehat z)+\dist(\widehat x',\widehat z'))^M\lesssim\dist(x,x'),
\end{equation}
for some $\widehat z,\widehat z'\in\widehat Z$ such that $\pi(\widehat z)=\pi(\widehat z').$

Now, let $v$ be as in the proposition and $\widehat v=v\circ\pi$ denote its lift. Taking the same notation as above we have
$$|v(x)-v(x')|=|\widehat v(\widehat x)-\widehat v(\widehat x')|\leq |\widehat v(\widehat x)-\widehat v(\widehat z)|+|\widehat v(\widehat z')-\widehat v(\widehat x')|,$$
since $\pi(\widehat z)=\pi(\widehat z')$ which implies that $\widehat v(\widehat z)=\widehat v(\widehat z').$ Therefore, the assumption on $\widehat v$ implies that
$$|v(x)-v(x')|\leq K(\dist(\widehat x,\widehat z)^\alpha+\dist(\widehat z',\widehat x')^\alpha),$$
and finally \eqref{eq lemme coeff holder} gives
$$|v(x)-v(x')|\leq aK\dist(x,x')^{\alpha/M},$$
where $a>0$ depends only on $\alpha,$ $L$ and $\pi.$
\end{proof}
\begin{corollary}\label{cor holder continuite pour les c-holo}
For every compact $L$ of $Z,$ there exists $0<\alpha\leq1$ such that every continuous weakly holomorphic function on $Z$ is $\alpha$-Hölder continuous on $L.$ Moreover, every uniformly bounded family of such functions is uniformly $\alpha$-Hölder continuous on $L.$
\end{corollary}
\begin{proof}
Recall that a continuous function on $Z$ is weakly holomorphic if it is holomorphic on the regular part of $Z.$ The result is known if $Z$ is smooth with $\alpha=1.$ Therefore, in general, it is enough to apply Proposition \ref{pro controle des coeff holder} to a desingularization of $Z.$
\end{proof}

In Proposition \ref{pro loja2}, we will need a similar result in a local setting but with a uniform control of the constants. It is the aim of the two following results.

\begin{proposition}\label{pro controle local des coeff holder}
Let $Z,$ $\widehat Z,$ $\pi$ and $L$ be as in Proposition \ref{pro controle des coeff holder}. Let $v$ be a function defined on a ball $B_Z(y,r)\subset L$ with $0<r\leq1/2.$ Assume that $\widehat v=v\circ\pi$ is $(K,\alpha)$-Hölder continuous. Then, there exist constants $0<\alpha'\leq 1,$ $a>0$ and $N\geq1,$ independent of $v,$ $y$ and $r$ such that $v$ is $(aK,\alpha')$-Hölder continuous on $B_Z(y,r^N).$
\end{proposition}
\begin{proof}
The proof is the same as that of Proposition \ref{pro controle des coeff holder} except that we have to check that if $x$ and $x'$ are in $B_Z(y,r^N)$ with $N\geq1$ large enough, then $\widehat v$ is well-defined at the points $\widehat z$ and $\widehat z'$ defined in \eqref{eq lemme coeff holder}.

Since $\pi$ is holomorphic, we have $\dist(x,\pi(\widehat z))\lesssim\dist(\widehat x,\widehat z).$ Therefore, by \eqref{eq lemme coeff holder} we have
$$\dist(x,\pi(\widehat z))\lesssim \dist(\widehat x,\widehat z)+\dist(\widehat x',\widehat z')\lesssim \dist(x,x')^{1/M}.$$
Hence, if $N$ is large enough, then $\widehat z$ and $\widehat z'$ belong to $\pi^{-1}(B_Z(y,r)).$ Then, the result follows as in Proposition \ref{pro controle des coeff holder}.
\end{proof}

\begin{corollary}\label{cor holder continuite local pour les c-holo}
 For every compact $L$ of $Z,$ there are constants $0<\alpha\leq1,$ $K>0$ and $N\geq1$ such that if $v$ is a continuous weakly holomorphic function on $B_Z(y,r)\subset L$ with $|v|\leq1$ then $v$ is $(K/r,\alpha)$-Hölder continuous on $B_Z(y,r^N).$
\end{corollary}
\begin{proof}
Let $\pi:\widehat Z\to Z$ be a desingularization. Let $\widehat z$ be in $\pi^{-1}(B_Z(y,r/2)).$ Since $\pi$ is holomorphic, there is $a>0$ such that $B_{\widehat Z}(\widehat z,r/a)$ is contained in $\pi^{-1}(B_Z(y,r)).$ Therefore, by Cauchy's inequality, $\widehat v$ is $a/r$-Lipschitz on $\pi^{-1}(B_Z(y,r/2)).$ Hence, the result follow from Proposition \ref{pro controle local des coeff holder}.
\end{proof}

\section{Volume estimate for endomorphisms}\label{sec vol}
The multiplicities of an endomorphism $f$ are strongly related to volume estimates which were used successfully to solve equidistribution problems. In what follows, we generalize Lojasiewicz type inequalities obtained in \cite{fs-cdhd2}, \cite{ds-equi} and \cite{ds-speed} to analytic sets, possibly singular. The aim is to control the size of a ball under iterations of $f$ in an invariant analytic set. Singularities, in particular the points where the analytic sets are not locally irreducible, lead to technical difficulties.

In this section, $X$ always denotes an irreducible analytic set of a smooth manifold. In order to avoid some problems related to the local connectedness of analytic sets, instead of the distance induced by an embedding of $X,$ we consider the distance $\rho$ defined by paths in $X.$ Namely, if $x,y\in X$ then $\rho(x,y)$ is the length of the shortest path in $X$ between $x$ and $y.$ These two distances on $X$ are related by the following result (see e.g. \cite{bm-subana}).
\begin{theorem}\label{th relation entre les distances sur X}
 Let $K$ be a compact subset of $X.$ There exists a constant $r>0$ such that for all $x,y\in K$ we have
$$\dist(x,y)\leq\rho(x,y)\lesssim\dist(x,y)^r.$$
\end{theorem}
The first step to state volume estimate is the following result.
\begin{proposition}\label{pro loja1}
 Let $\Gamma\subset \mathbb B\times \mathbb B$ and $X\subset \mathbb B$ be two analytic subsets with $X$ locally irreducible and such that the first projection $\pi: \mathbb B\times \mathbb B\to \mathbb B$ defines a ramified covering of degree $m$ from $\Gamma$ to $X.$ There exist constants $a>0$ and $b\geq1$ such that if $x,y\in X\cap \mathbb B_{1/2}$ then we can write
$$\pi^{-1}(x)\cap\Gamma=\{x^1,\ldots,x^m\} \textrm{ and } \pi^{-1}(y)\cap\Gamma=\{y^1,\ldots,y^m\},$$
with $\dist(x,y)\geq a\dist(x^i,y^i)^{bm}.$ Moreover, $a$ increases with $m$ but is independent of $\Gamma$ and $b$ depends only on $X.$
\end{proposition}
\begin{proof}
By Theorem \ref{th relation entre les distances sur X}, to establish the proposition, we can replace $\dist(x,y)$ by $\rho(x,y)$ on $X.$ For $w\in X_{Reg},$ we can define the $j$-th Weierstrass polynomial, $k+1\leq j\leq 2k,$ on $t\in \mathbb C$
$$P_j(t,w)=\prod_{z\in\pi^{-1}(w)\cap\Gamma}(t-z_j)=\sum_{l=0}^ma_{j,l}(w)t^l,$$
where $z=((z_1,\ldots,z_k),(z_{k+1},\ldots,z_{2k}))\in \mathbb B\times \mathbb B.$
The coefficients $a_{j,l}$ are holomorphic on $X_{Reg}$ and uniformly bounded by $m!$ since $\Gamma\subset \mathbb B\times \mathbb B.$ As $X$ is locally irreducible, they can be extended continuously to $X$ (see e.g. \cite{de-book}). It gives a continuous extension of the polynomials $P_j$ to $X$ with $P_j(z_j,\pi(z))=0$ if $z$ is in $\Gamma.$ Moreover, by Corollary \ref{cor holder continuite pour les c-holo} there exists $\alpha>0$ such that the coefficients $a_{j,l}$ are uniformly $\alpha$-Hölder continuous on $X\cap \mathbb B_{3/4}$ with respect to $\rho.$

We claim that there is a constant $a>0$ such that if $x,y\in X\cap \mathbb B_{1/2}$ and $x'\in \mathbb B$ with $\widetilde x:=(x,x')\in\Gamma$ then there is $\widetilde y\in \pi^{-1}(y)\cap\Gamma$ with
$$\rho(x,y)\geq a\dist(\widetilde x,\widetilde y)^{bm},$$
where $b=1/\alpha.$ From this, the result follows exactly as in the end of the proof of \cite[Lemma 4.3]{ds-equi}.

It remains to prove the claim. Fix $c>0$ large enough and let $r=\rho(x,y).$ We can assume that $r$ is non-zero and sufficiently small, otherwise the result is obvious. Since the covering has degree $m$, we can find an integer $2\leq l\leq4km$ such that for all $k+1\leq j\leq 2k$, no root of the polynomial $P_j(t,x)$ satisfies
$$c(l-1)r^{1/bm}\leq |\widetilde x_j-t|\leq c(l+1)r^{1/bm}.$$
It gives \textit{a security ring} over $x$ which does not intersect $\Gamma.$ Using the regularity of $P_j,$ it can be extend to a neighborhood of $x.$
More precisely, for $\theta\in\mathbb R$ we define $\xi_j=\widetilde x_j+cle^{i\theta}r^{1/bm}$ and
$$G_{j,c,\theta}(w)=c^{-m+1}P_j(\xi_j,w).$$
The choice of $l$ implies that
$$|G_{j,c,\theta}(x)|=c^{-m+1}|P_j(\xi_j,x)|\geq c^{-m+1}\prod_{z\in\pi^{-1}(x)}|\xi_j-z_j|\geq cr^{\alpha}.$$
Moreover, we deduce from the properties on the coefficients $a_{j,l}$ that the functions $G_{j,c,\theta}$ are uniformly $\alpha$-Hölder continuous on $X\cap \mathbb B_{3/4}$ with respect to $(j,c,\theta).$ Hence, if $c$ is large enough, they do not vanish on $B:=\{z\in X\,|\,\rho(x,z)<2r\}$ which contains $y.$

It implies that $P_j(t,w)\neq 0$ if $w$ is in $B$ and $|t-\widetilde x_j|=lcr^{1/bm}.$ Therefore, if we  denote by $\Sigma$ the boundary of the polydisc of center $x'$ and of radius $lcr^{1/bm},$ we have $\Gamma\cap (B\times \Sigma)=\varnothing.$ Hence, since $B$ is connected and contains $y,$ by continuity there is a point $\widetilde y$ in $\pi^{-1}(y)\cap\Gamma$ with $|\widetilde x_j-\widetilde y_j|\leq lcr^{1/bm}.$ This completes the proof of the claim.
\end{proof}
\begin{remark}\label{rem loja}
 Our proof shows that if we assume that $m\leq m_0$ for a fixed $m_0,$ then $a$ depends only on the Hölder constants $(K,\alpha)$ associated with $a_{j,l}.$ Furthermore, if $\alpha$ is fixed then $a$ is proportional to $K.$ Note that without assumption on the constants $a$ and $b$, Proposition \ref{pro loja1} can be deduced from Theorem \ref{th iné de Loja classique pour les fonction}.
\end{remark}

The control of multiplicities of the covering gives the following more precise result.
\begin{proposition}\label{pro loja2}
Let $X,$ $\Gamma$ and $\pi$ be as above. Let $Z\subset X$ be a proper analytic subset. Assume that the multiplicity at each point of $\pi^{-1}(x)$ is at most equal to $s$ if $x\in X\setminus Z.$ Then, there exist constants $a>0,$ $b\geq1$ and $N\geq1$ such that for all $0<t\leq1/2$ and $x,y\in X\cap \mathbb B_{1/2}$ with $\dist(x,Z)>t$ and $\dist(y,Z)>t,$ we can write
$$\pi^{-1}(x)\cap\Gamma=\{x^1,\ldots,x^m\} \textrm{ and } \pi^{-1}(y)\cap\Gamma=\{y^1,\ldots,y^m\},$$
with $\dist(x,y)\geq at^N\dist(x^i,y^i)^{bs}.$
\end{proposition}
\begin{proof}
Let $t>0$ and $x\in X\cap \mathbb B_{1/2}$ with $\dist(x,Z)>t.$ We want to find a neighborhood $B\subset X$ of $x$ such that each component of $\Gamma\cap \pi^{-1}(B)$ defines a ramified covering of degree at most equal to $s$ over $B.$

We construct an analytic set $Y$ associate to the multiplicities of $\Gamma.$ Namely, first define $Y'\subset \Gamma^{s+1}$ by $\{z\in\Gamma^{s+1}\,|\, \pi\circ\tau_i(z)=\pi\circ\tau_j(z),\ 1\leq i,j\leq s+1\},$ where $\tau_i,$ $1\leq i\leq s+1,$ are the canonical projections of $\Gamma^{s+1}$ onto $\Gamma.$ For $i\neq j$ we set $A_{i,j}=\{z\in Y'\,|\, \tau_i(z)=\tau_j(z)\}.$ Then, $Y$ is defined by $\overline{Y'\setminus \cup_{i\neq j}A_{i,j}}.$ The map $\pi_1=\pi\circ\tau_1: Y\to X$ defines a ramified covering. If $x$ is generic in $X$ then a point $z\in Y$ over $x\in X$ represents a family of $(s+1)$ distinct points in $\pi^{-1}(x)\cap\Gamma.$

If $\pi'$ denotes the second projection of $\mathbb B\times \mathbb B$ onto $\mathbb B,$ consider the map $h:Y\to \mathbb C^{(s+1)^2}$ define by
$$h(z)=(\pi'\circ\tau_i(z)-\pi'\circ\tau_j(z))_{1\leq i,j\leq s+1}.$$
By construction, $h(z)=0$ means precisely that there is a point in $\Gamma$ with multiplicity greater than $s$ over $\pi_1(z).$ It implies that $\pi_1(h^{-1}(0))\subset Z.$ Hence, Theorem \ref{th iné de Loja classique pour les fonction} implies there is a constant $M>0$ such that if $\pi(z)\in X\cap \mathbb B_{1/2}$ then
\begin{equation}\label{eq dans loja2}
 \|h(z)\|\gtrsim\dist(z,h^{-1}(0))^M\gtrsim\dist(\pi_1(z),Z)^M.
\end{equation}

Let $a,b>0$ be the constants in Proposition \ref{pro loja1}. As in the proof of that proposition, we use the distance function $\rho$ on $X.$ Fix $\gamma>0$ small enough and set $B=\{z\in X\,|\, \rho(x,z)<\gamma t^{Mbm})\}.$ For $\widetilde x=(x,x')$ in $\Gamma,$ we can choose $2\leq l\leq 8m$ such that $\pi^{-1}(x)\cap\Gamma$ do not intersect the ring
$$r(l-2)\leq \|\widetilde x-w\|\leq r(l+2),$$
where $r=(a^{-1}\gamma)^{1/bm}t^M.$ Hence, if $H$ denotes the ball of center $x'$ and of radius $rl,$ we have $\dist(\pi^{-1}(x)\cap \Gamma,B\times\partial H)> r.$ Therefore, by Proposition \ref{pro loja1} $\Gamma\cap B\times\partial H=\varnothing.$ It assures that $\pi$ is proper on $\Gamma\cap B\times H$ and then defines a ramified covering.

Moreover, this covering has degree at most equal to $s.$ Otherwise, according to the radius of $H,$ we have
$$\min_{z\in\pi_1^{-1}(x)}\|h(z)\|\leq (s+1)^22rl,$$
which is in contradiction with \eqref{eq dans loja2} if $\gamma$ is small enough, since $\dist(x,Z)>t$ and $r=(a^{-1}\gamma)^{1/bm}t^M.$

Now, we want to apply Proposition \ref{pro loja1} to this covering. But, in order to control the constants, we have to reduce $B.$ Indeed, according to Remark \ref{rem loja}, the constants of that proposition depend only on the Hölder continuity of the coefficients $a_{j,l}$ (and on the degree of the covering). These coefficients are bounded continuous weakly holomorphic functions defined on $B$ then by Corollary \ref{cor holder continuite local pour les c-holo} with $L=\overline{X\cap\mathbb B_{3/4}},$ they are $(Kt^{-M'},\alpha)$-Hölder continuous on $B(x,t^{M'})$ for some $M'\geq1$ large enough, $0<\alpha\leq1$ and $K>0$ independent of $x$ and $t.$ Therefore, after a coordinates dilation by $t^{-M'}$ at $x,$ we can apply Proposition \ref{pro loja1} with the same exponent $b$ and the second constant proportional to some power of $t.$ Finally, let $y\in X.$ We can assume that $\rho(x,y)\leq t^{M'}/2,$ otherwise the proposition is obvious. Hence, the previous observation implies that
$$\pi^{-1}(x)\cap\Gamma\cap B\times H=\{x^1,\ldots,x^s\} \textrm{ and } \pi^{-1}(y)\cap\Gamma\cap B\times H=\{y^1,\ldots,y^s\},$$
with $\rho(x,y)\geq a't^N\dist(x^i,y^i)^{bs}$ where $a'>0,b\geq1$ and $N\geq1$ are independent of $x$ and $t.$ More precisely, we can write $N=M'+N',$ where the contribution in $t^{M'}$ comes from the dilation and that in $t^{N'}$ comes from the estimate on Hölder continuity. The construction can be applied to each component of $\Gamma\cap\pi^{-1}(B).$ This gives the result.
\end{proof}

We now consider the dynamical context. Assume that $X\subset\mathbb P^k$ is an irreducible analytic set of dimensions $l$ which is invariant by $f.$ Denote by $g$ the restriction of $f$ to $X.$ For $x$ in $X,$ we define \textit{the local multiplicity} of $g$ at $x$ as the maximal number of points in $g^{-1}(z)$ which are near $x$ for $z\in X$ close enough to $g(x).$ The local multiplicity is smaller than \textit{the topological degree} i.e. the number of points in $g^{-1}(x)$ for $x$ generic in $X.$ In our case, the topological degree of $g$ is equal to $d^l,$ see \cite{ds-equi}.

There exists a finite covering $(U_i)_{i\in I}$ of $X$ by open subsets of $\mathbb P^k$  such that $X\cap U_i$ can by decompose into locally irreducible components. Hence, we can apply Proposition \ref{pro loja2} to the graph of $g$ over each component of $X\cap U_i.$ It gives the following corollary.
\begin{corollary}\label{cor dist loja}
Let $\eta>1$ and $Z\subset X$ be a proper analytic subset. Assume that the local multiplicity of $g$ is less than $\eta$ outside $g^{-1}(Z).$ Then there are constants $a>0,$ $b\geq1$ and $N\geq1$ such that if $0<t<1$ and $x,y$ are two points outside $Z_t$ where $X$ is locally irreducible, then we can write
\begin{equation*}
 g^{-1}(x)=\{x^1,\ldots,x^{d^l}\}\textrm{ and } g^{-1}(y)=\{y^1,\ldots,y^{d^l}\}
\end{equation*}
with $\dist(x,y)\geq at^{N}\dist(x^j,y^j)^{b\eta}.$
\end{corollary}
From this, we obtain the following size estimate for image of balls which is crucial in the proof of our main result.
\begin{corollary}\label{cor loja}
 Let $\delta>1$ and $E\subset X$ be a proper analytic subset. Denote by $\widetilde E$ the preimage of $E$ by $g.$ Assume that the local multiplicity is less than $\delta$ outside $\widetilde E.$ There exist constants $0<A\leq1,$ $b\geq1$ and $N\geq 1$ such that if $0<t\leq1/2,$ $r<t/2$ and $x\in X\setminus \widetilde E_t,$ then $g(B_{X}(x,r))$ contains a ball of radius $At^Nr^{b\delta}.$ Moreover, $b$ depends only on $X.$
\end{corollary}
\begin{proof}
Fix $t>0$ and $r<t/2.$ As in the proof of Lemma \ref{le vol de ball par eclat} with $\widehat Y=X_{Sing}\cup g^{-1}(X_{Sing}),$ possibly after replacing $r$ by $cr^l$ for some $c>0,$ we can assume that $B_{X}(x,r)$ and $g(B_{X}(x,r))$ are contained in $X_{Reg}.$ The local multiplicity of $f$ on $X$ is bounded $d^{l}.$ So, there exists $2\leq i\leq 4d^{l},$ such that the ring $\{\frac{r(i-1)}{4d^{l}+1}\leq \dist(x,x')\leq\frac{r(i+1)}{4d^{l}+1}\}$ contains no preimage of $g(x).$ Thus, if $x'\in\partial B_{X}(x,r\frac{i}{4d^{l}+1}),$ then 
\begin{equation*}
 \dist(x',g^{-1}(g(x)))\geq \frac{r}{4d^{l}+1}.
\end{equation*}
Moreover, we can apply Corollary \ref{cor dist loja} with $\eta=d^{l}$ and $Z=\varnothing.$ Hence, there exists $a,b>0$ such that $g(\overline{B_{X}(x,r\frac{i}{4d^{l}+1})})\subset X\setminus E_{a(t/2)^{bd^{l}}}.$ Therefore, we can apply once again Corollary \ref{cor dist loja} with $\eta=\delta,$ $Z=E$ and $a(t/2)^{bd^l}$ instead of $t.$ We get, for some constants $a'>0$ and $N_0\geq 1$
\begin{equation*}
 \dist(g(x'),g(x))\geq a'\left(a\left(\frac{t}{2}\right)^{bd^{l}}\right)^{N_0}\left(\frac{r}{4d^{l}+1}\right)^{b\delta},
\end{equation*}
and, since $g$ is an open mapping near $x$
\begin{equation*}
B_{X}(g(x),At^Nr^{b\delta})\subset g(B_{X}(x,r))
\end{equation*}
with $A=\frac{a^{N_0}a'}{2^{N_0bd^{l}}(4d^{l}+1)^{b\delta}}$ and $N=N_0bd^{l}.$
\end{proof}
\begin{remark}
 When $X$ is smooth, the ball in $g(B_X(x,r))$ can be chosen centered at $g(x).$
\end{remark}

\section{Psh functions and exponential estimates}\label{sec ho}
We refer to \cite{de-book} for basics on currents and plurisubharmonic (psh for short) functions. Let $T$ be a positive closed $(1,1)$-current of mass $1$ on $\mathbb P^k$ with continuous local potentials. Let us recall briefly the associated notions of psh and weakly psh (wpsh for short) modulo $T$ functions introduced in \cite{ds-equi}. 

Let $Y$ be an analytic space. A function $v:Y\to \mathbb R\cup\{-\infty\}$ is \textit{wpsh} if it is psh on $Y_{Reg}$ and for $y$ in $Y,$ we have $v(y)=\limsup v(z)$ with $z\in Y_{Reg}$ and $z\to y.$ These functions coincide with psh functions if $Y$ is smooth. On compact spaces, the notion is very restrictive. However, if $X$ is an analytic subset of $\mathbb P^k,$ we have the more flexible notion of \textit{wpsh modulo $T$ function} on $X.$ Locally, it is the difference of a wpsh function on $X$ and a potential of $T.$ If $X$ is smooth, we say that the function is psh modulo $T.$

Note that the restriction of a psh modulo $T$ function to an analytic subset is either wpsh modulo $T$ or equal to $-\infty$ on an irreducible component. If $u$ is wpsh modulo $T$ on $X$ then $dd^c(u[X])+T\wedge[X]$ is a positive closed current supported on $X.$ On the other hand, if $S$ is a positive closed $(1,1)$-current on $\mathbb P^k$ with mass $1,$ there is a psh modulo $T$ function $u$ on $\mathbb P^k,$ unique up to a constant, such that $S=dd^cu+T.$

These notions bring good compactness properties which permit to obtain uniform estimates. We have the following statements established in \cite{ds-equi}.
\begin{proposition}
 Let $(u_n)$ be a sequence of wpsh modulo $T$ functions on $X,$ uniformly bounded from above. Then there is a subsequence $(u_{n_i})$ satisfying one of the following properties:
\begin{itemize}
\item There is an irreducible component $Y$ of $X$ such that $(u_{n_i})$ converges uniformly to $-\infty$ on $Y\setminus X_{Sing}.$
\item $(u_{n_i})$ converges in $L^p(X)$ to a wpsh modulo $T$ function $u$ for every $1\leq p<+\infty.$
\end{itemize}
In the last case, $\limsup u_{n_i}\leq u$ on $X$ with equality almost everywhere.
\end{proposition}
It implies the following lemma.
\begin{lemma}\label{le wpsh borne sur ss ens}
 Let $\cali G$ be a family of psh modulo $T$ functions on $\mathbb P^k$ uniformly bounded from above. Assume that each irreducible component of $X$ contains an analytic subset $Y$ such that the restriction of $\cali G$ to $Y$ is bounded in $L^1(Y).$ Then, the restriction of $\cali G$ to $X$ is bounded in $L^1(X).$
\end{lemma}

A classical result of Hörmander \cite{ho-book} gives a uniform bound to $\exp(-u)$ in $L^1(\mathbb B_{1/2})$ for $u$ in a class of psh functions in the unit ball of $\mathbb C^k.$ Similar estimates can be obtained for compact families of quasi-psh functions. From now, we assume that $T$ has $(K,\alpha)$-Hölder continuous local potentials, with $0<\alpha\leq 1$ and $K>0.$ In the rest of this section, we establish exponential estimates for psh modulo $T$ functions in different situations. A key observation in our approach is that Hölder continuity allows us to work with non-compact families. In the sequel, we will apply these estimates to $T$ the Green current associate to $f.$ They allow us to control the volume of some sublevel sets of potentials of currents near exceptional sets.

As a consequence of classical Hörmander's estimate, we have the following lemma which will be of constant use. Here, $\nu$ denotes the standard volume form on $\mathbb C^k$ and $T$ is seen as a fixed current on the unit ball $\mathbb B$ of $\mathbb C^k.$ We assume that its admits a potential $g$ which is $(K,\alpha)$-Hölder continuous on $\mathbb B.$
\begin{lemma}\label{le Ho1: avec valeur centrale}
 Let $v$ be a psh modulo $T$ function in $\mathbb B_t$ with $v\leq 0$ and $v(0)>-\infty.$ Let $0<s<-v(0)^{-1}$ and $t>0$ such that $Kt^\alpha\leq s^{-1}.$ There is a constant $c>0$ independent of $v,$ $s$ and $t$ such that
\begin{equation}\label{eq Ho1: avec valeur centrale}
 \int_{\mathbb B_{t/2}}\exp(-\frac{sv}{2})\nu\leq ct^{2k}.
\end{equation}
\end{lemma}
\begin{proof}
As $v$ is psh modulo $T$, we have $v=v'-g$ with $v'$ psh. We set $\widetilde v(z)=v'(z)-g(0)-Kt^\alpha.$ Then $\widetilde v$ is psh in $\mathbb B_t,$ $\widetilde v(0)=v(0)-Kt^\alpha\geq -2s^{-1}$ and $\widetilde v\leq v\leq 0$ because $g(z)-g(0)\leq Kt^\alpha$ on $\mathbb B_t.$ By \cite[Theorem 4.4.5]{ho-book} there exists $c>0$ such that
\begin{equation*}
 \int_{\mathbb B_{1/2}}\exp(-\frac{s\widetilde v(tz)}{2})\nu\leq c,
\end{equation*}
thus, by a change of variables $z\mapsto tz,$ we get
\begin{equation*}
  \int_{\mathbb B_{t/2}}\exp(-\frac{sv}{2})\nu\leq\int_{\mathbb B_{t/2}}\exp(-\frac{s\widetilde v}{2})\nu\leq ct^{2k}.
\end{equation*}
\end{proof}

For the rest of the section, $X$ always denotes an irreducible analytic subset of $\mathbb P^k$ of dimension $l$ and $v$ is a psh modulo $T$ function in $\mathbb P^k$ with $v\leq0.$ In Section \ref{sec equi} we will extend the previous result to the neighborhood of $X,$ where the condition at $0$ is replaced by an integrability condition on $X.$ For this purpose, we have to control the size of sublevel sets of $v$ in $X.$ This is the aim of the following global result.
\begin{lemma}\label{le Ho2: avec valeur L^p}
 For $X$ as above, there exists $q_0\geq1$ with the following property. For $q>q_0$ we set $\epsilon=2lq_0/q\alpha$ and take $M>0$ and $s\geq 1$ such that $s^{1+\epsilon}\|v\|_{L^q({X})}\leq M.$ Then, there exist constants $a,c>0$ independent of $v,$ $q$ and $s$ such that
\begin{equation}\label{eq Ho2: avec valeur L^p}
 \int_{X}\exp(-asv)\omega^l\leq c.
\end{equation}
If $X$ is smooth, we can choose $q_0=1.$
\end{lemma}
\begin{proof}
First, assume that $X$ is a compact smooth manifold with a volume form $\eta.$ Since $X$ has dimension $l,$ for $t>0$ we can cover it by balls $(B_i)_{i\in I}$ with $B_i:=B_X(x_i,t)$ and such that $|I|\leq c't^{-2l}$ for some $c'>0.$ Let $t=s^{-1/\alpha}.$
As above, in each ball $B_X(x_i,2t)$ we can write $v=v_i'-g_i,$ where $g_i$ is a local potential of $T.$ Using local charts at $x_i,$ we can identify $B_X(x_i,2t)$ with $\mathbb B_{2t}$ in $\mathbb C^l.$ We consider $\widetilde v_i(z)=s(v_i'(tz)-g_i(0)).$ These functions are psh in $\mathbb B_2.$ We show that they belong to a compact family, independent of $v$ and $s.$ Using a change of variables $z\mapsto tz$ and Hölder's inequality, we get
\begin{align*}
 \|\widetilde v_i\|_{L^1(\mathbb B_2)}&\leq\int_{\mathbb B_2}s|v(tz)|\nu + \int_{\mathbb B_2}s|g_i(tz)-g_i(0)|\nu\\
&\leq st^{-2l}\|v\|_{L^q(X)}|\mathbb B_2|^{1/p}t^{2l/p}+2^\alpha K|\mathbb B_2|\\
&=s^{1+\epsilon}\|v\|_{L^q(X)}|\mathbb B_2|^{1/p}+2^\alpha K|\mathbb B_2|\\
&\leq M|\mathbb B_2|^{1/p}+2^\alpha K|\mathbb B_2|\leq M',
\end{align*}
where $p$ is the conjugate of $q$, $|\mathbb B_2|$ is the volume of $\mathbb B_2$ and $M'$ is a positive constant. The family $\cali U=\{u\in PSH(\mathbb B_2)\,|\,\|u\|_{L^1(\mathbb B_2)}\leq M'\}$ is compact so there exists a constant $a>0$ such that $\|\exp(-au)\|_{L^1(\mathbb B)}$ is uniformly bounded for all $u\in\cali U.$ Therefore, for $i\in I$
\begin{equation*}
 \int_{\mathbb B_t}\exp(-as(v_i'(z)-g_i(0)))\nu\lesssim t^{2l}.
\end{equation*}
Moreover, the Hölder continuity implies that $-sv(z)\leq K-s(v_i'(z)-g_i(x_i))$ in $B_i.$ Hence, since $(B_i)_{i\in I}$ is a covering of $X$ we obtain
\begin{align*}
  \int_{X}\exp(-asv)\eta&\leq\sum_{i\in I}\int_{B_i}\exp(-asv)\eta \\
&\leq \sum_{i\in I}\int_{B_i}\exp(a(K-s(v_i'(z)-g_i(x_i))))\eta \\
&\lesssim\sum_{i\in I}t^{2l}\leq c'.
\end{align*}
This implies the lemma if $X$ is smooth with $q_0=1.$

In the general case, we consider a desingularization $\pi:\widehat X\to X$ with a volume form $\eta$ on $\widehat X.$ The map $\pi$ is surjective, then by Lemma \ref{le remonté d'estimée par surjection}, there exists $q_0\geq1$ such that
$$\|\widehat v\|_{L^{q/q_0}(\widehat X,\eta)}\lesssim\|v\|_{L^{q}(X,\omega^l)}.$$
Moreover, $\pi^*(T)$ possesses $\alpha$-Hölder local potentials and $\widehat v\leq 0$ is psh modulo $\pi^*(T).$ Therefore, this choice of $q_0$ allows us to apply the lemma on $\widehat X$ and get 
$$\int_{\widehat X}\exp(-as\widehat v)\eta\leq c.$$
The result follows since
$$\int_{X}\exp(-asv)\omega^l=\int_{\widehat X}\exp(-as\widehat v)\pi^*(\omega^l)\leq \|h\|_\infty \int_{\widehat X}\exp(-as\widehat v)\eta,$$
where we write $\pi^*(\omega^l)=h\eta.$
\end{proof}
The following estimate is a consequence of Lemma \ref{le Ho1: avec valeur centrale} and is related to the geometry of sublevel sets of psh modulo $T$ functions. In Section \ref{sec equi}, it will establish the existence of balls where we can apply our volume estimates.
\begin{lemma}\label{le Ho3: geometrie des sous niveau}
For $s\geq2$ set $F_s=\{x\in X\,|\, v(x)\leq -s^{-1}\}.$ There are constants $\beta,c>0$ independent of $v$ and $s$ such that  if $F_s$ contains no ball of radius $s^{-\beta}$ then
\begin{equation*}
 \int_{X}\exp(-\frac{sv}{2})\omega^l\leq c.
\end{equation*}
\end{lemma}
\begin{proof}
 We first consider the case where $X$ is smooth. Let $t=4^{-1}(Ks)^{-1/\alpha}.$ As in the proof of the previous lemma, we cover $X$ by balls $(B_i)_{i\in I}$ of radius $t$ with $|I|\leq c't^{-2l},$ $c'>0.$ Assume there is no ball of radius $t$ in $F_s.$ Hence, for each $i\in I$ there exists $x_i$ in $B_i$ such that $v(x_i)>-s^{-1}.$ The balls $B'_i$ of center $x_i$ and of radius $2t$ cover $X.$ Thus
\begin{equation*}
 \int_{X}\exp(-\frac{sv}{2})\omega^l\leq\sum_{i\in I}\int_{B'_i}\exp(-\frac{sv}{2})\omega^l.
\end{equation*}
But, $s<-v(x_i)^{-1}$ and $K(4t)^\alpha\leq s^{-1}$ therefore we can apply Lemma \ref{le Ho1: avec valeur centrale} on each ball
\begin{equation*}
 \int_{B'_i}\exp(-\frac{sv}{2})\omega^l\lesssim t^{2l}.
\end{equation*}
Hence, we get
\begin{equation*}
 \int_{X}\exp(-\frac{sv}{2})\omega^l\lesssim\sum_{i\in I}t^{2l}\leq c',
\end{equation*}
which gives the result when $X$ is smooth with $\beta>1/\alpha$ such that $s^{-\beta}<t.$

If $X$ is singular, we consider a desingularization $\pi:\widehat X\to X.$ By Lemma \ref{le vol de ball par eclat}, there exists $N\geq1$ such that the image of a ball of radius $r$ under $\pi$ contains a ball of radius $r^{N}.$ Hence, if $\beta$ is large enough, the hypothesis on $F_s$ assures there is no ball of radius $t$ in $\widehat F_s=\pi^{-1}(F_s).$ Then, we can apply the lemma to $\widehat v=v\circ\pi$ which is psh modulo $\pi^*(T).$ We get
$$\int_X\exp(-\frac{sv}{2})\omega^l=\int_{\widehat X}\exp(-\frac{s\widehat v}{2})\pi^*(\omega^l)\leq c,$$
for some $c>0,$ since $\pi^*(\omega^l)$ is smooth.
\end{proof}

\section{Exceptional sets}\label{sec multi}
Let $f$ be an endomorphism of $\mathbb P^k$ of algebraic degree $d\geq2.$ The aim of this section is to construct two families $\cali A_\lambda$ and $\cali B_\lambda$ of analytic sets where the iterate sequence $f^n$ has important local multiplicities. Let $X\subset\mathbb P^k$ be an irreducible invariant analytic set. Define $\kappa_{X,n}(x),$ or simply $\kappa_n(x)$ if no confusion is possible, as the local multiplicity of $f_{|X}^n$ at $x.$ It is \textit{a sub-multiplicative cocycles}, namely it is upper semi-continuous for the Zariski topology on $X,$ $\min_X\kappa_{n}=1$ and for any $m,n\geq 0$ and $x\in X$ we have the following sub-multiplicative relation
\begin{equation*}
 \kappa_{n+m}(x)\leq\kappa_{m}(f^n(x))\kappa_{n}(x).
\end{equation*}
The inequality may be strict when $X$ is singular. Define
\begin{equation*}
 \kappa_{-n}(x):=\max_{y\in (f_{|X})^{-n}(x)}\kappa_{n}(y).
\end{equation*}
We recall the following theorem of Dinh \cite{d-ana}, see also \cite{favre-multi}.
\begin{theorem}\label{th ens excep}
The sequence of functions $\kappa_{-n}^{1/n}$ converges pointwise to a function $\kappa_{-}.$ Moreover, for every $\lambda>1,$ the level set $E_{\lambda}(X)=\{\kappa_{-}\geq \lambda\}$ is a proper analytic subset of $X$ which is invariant under $f_{|X}.$ In particular, $\kappa_{-}$ is upper semi-continuous in the Zariski sense.
\end{theorem}
For a generic endomorphism of $\mathbb P^k,$ $E_\lambda(\mathbb P^k)$ is empty. In this case, Theorem \ref{th main} is already know in all bidegrees \cite{ds-superpot}. In our proof, we will proceed by induction, proving the the exponentially fast convergence on $X$ if it is already established on each irreducible component of $E_\lambda(X).$ But, even if $E_\lambda(X)$ is invariant, its irreducible components are periodic and not invariant in general. Therefore, if $X$ is only periodic, we define $E_\lambda(X)$ in the same way, replacing $f$ by $f^p$ and $\lambda$ by $\lambda^p,$ where $p$ is a period of $X.$ By Theorem \ref{th ens excep}, this definition is independent of the choice of $p.$

Fix $1<\lambda<d.$ Define the family $\cali B_\lambda$ of exceptional sets as follows. First, we set $\mathbb P^k\in\cali B_\lambda.$ If $X$ is in $\cali B_\lambda,$ we add to $\cali B_\lambda$ all irreducible components of $E_\lambda(X).$ This family is finite and since the functions $\kappa_-$ are upper semi-continuous in the Zariski sense, there exists $1<\delta<\lambda$ such that $\cali B_\lambda=\cali B_\delta,$ or equivalently $E_\lambda(X)=E_\delta(X)$ if $X\in\cali B_\lambda.$ This will give us some flexibility in order to obtain estimates using an induction process. 

As all elements of $\cali B_\lambda$ are periodic, they are invariant under some iterate $f^{n_0}.$ Let us remark that it is sufficient to prove Theorem \ref{th main} for an iterate of $f.$ Hence, we can assume that $n_0=1,$ replacing $f$ and $\lambda$ by $f^{n_0}$ and $\lambda^{n_0}.$ Dinh also proved that $\kappa_{n_1}<\delta^{n_1}$ outside $(f_{|X})^{-n_1}(E_\lambda(X))$ for some $n_1\geq1.$ Once again, we can assume that $n_1=1.$

The second family $\cali A_\lambda,$ that takes place in Theorem \ref{th main}, is defined as the set of minimal elements for the inclusion in $\cali B_\lambda.$ This family is not empty and each element of $\cali B_\lambda$ contains at least one element of $\cali A_\lambda.$ Note that no element of $\cali A_\lambda$ is contained in another one. These analytic sets play a special role in the next section, to start induction and to obtain compactness properties. When $\mathbb P^k$ is an element of $\cali A_\lambda,$ it is the only element in $\cali A_\lambda$ and the exceptional set is empty. Otherwise, define the exceptional set as the union of all the elements of $\cali A_\lambda.$

\section{Equidistribution speed}\label{sec equi}
This section is devoted to the proof of Theorem \ref{th main}. Fix an endomorphism $f$ of algebraic degree $d\geq 2$ of $\mathbb P^k,$ and denote by $T$ its Green current. Recall that $T$ is totally invariant i.e. $d^{-1}f^*(T)=T,$ and has $(K,\alpha)$-Hölder continuous local potentials for some $0<\alpha\leq1,$ $K>0.$

Fix $C>0$ and $1<\lambda<d,$ and let $\cali A_\lambda,$ $\cali B_\lambda$ be as in Section \ref{sec multi}. Define $\cali F_\lambda(C)$ as the family of psh modulo $T$ functions $v$ on $\mathbb P^k$ such that $\max_{\mathbb P^k}v=0$ and $\|v\|_{L^1(X)}\leq C$ for all $X\in\cali A_\lambda.$ By construction of $\cali A_\lambda,$ Lemma \ref{le wpsh borne sur ss ens} implies that $\cali F_\lambda(C)$ is compact for each $C>0.$  Moreover, if $X$ is an element of $\cali B_\lambda,$ the restriction of $\cali F_\lambda(C)$ to $X$ forms a family of wpsh modulo $T$ functions on $X$ which is relatively compact in $L^p(X)$ for every $1\leq p<+\infty.$

If $S$ is a positive closed $(1,1)$-current of mass $1,$ it is cohomologous to $T.$ Hence, there exists a unique psh modulo $T$ function $u$ on $\mathbb P^k$ such that $S=dd^cu+T$ and $\max_{\mathbb P^k}u=0.$ We call $u$ the dynamical potential of $S.$ As $T$ is totally invariant, the dynamical potential of $S_n=d^{-n}(f^n)^*(S)$ is $u_n=d^{-n}u\circ f^n.$

Since $S_n-T$ is a continuous linear operator on $\cali C^0(\mathbb P^k)$ whose norm is bounded, by interpolation theory between Banach spaces we have
\begin{equation*}
\|S_n-T\|_{\cali C^\beta}\lesssim \|S_n-T\|_{\cali C^2}^{\beta/2},
\end{equation*}
uniformly in $S$ and $n,$ see \cite{triebel-interpolation}. Consequently, in order to prove Theorem \ref{th main} we can assume that $\beta=2.$

Moreover, it is easy to see that $\|dd^c\phi\|_\infty\lesssim\|\phi\|_{\cali C^2}$ for $\phi$ in $\cali C^2(\mathbb P^k).$ Therefore,
\begin{align*}
 |\langle S_n-T,\phi\rangle|&=|\langle dd^cu_n,\phi\rangle|=|\langle u_n,dd^c\phi\rangle| \\
&\lesssim\|\phi\|_{\cali C^2}\|u_n\|_{L^1(\mathbb P^k)}.
\end{align*}
Hence, Theorem \ref{th main} is a direct consequence of the following theorem applied to $p=1$ and $X=\mathbb P^k.$
\begin{theorem}\label{th la version u_n du main}
For each $1\leq p<+\infty$ and $X\in\cali B_\lambda$ there exists a constant $A_{X,p}$ such that for all $u\in\cali F_\lambda(C)$ and $n\geq 0$ we have
\begin{equation*}
 \|u_n\|_{L^p(X)}\leq A_{X,p}\left(\frac{\lambda}{d}\right)^n,
\end{equation*}
where $u_n=d^{-n}u\circ f^n.$
\end{theorem}
As in Section \ref{sec multi}, we can assume that each element of $\cali B_\lambda$ is invariant by $f,$ and there is $1<\delta<\lambda$ satisfying the following properties for all $X$ in $\cali B_\lambda:$
\begin{itemize}
 \item $E_\lambda(X)=E_\delta(X),$
\item $\kappa_{X,1}<\delta$ outside $\widetilde E_\lambda(X)=(f_{|X})^{-1}(E_\lambda(X)).$
\end{itemize}
Let $X$ be an element $\cali B_\lambda$ of dimension $l$ and $\lambda_1>0$ with $\delta<\lambda_1<\lambda.$ Assume that Theorem \ref{th la version u_n du main} is true on each irreducible component of $E_\lambda=E_\lambda(X)$ for $\lambda_1$ and all $p\geq 1.$ To prove it on $X,$ we consider the sublevel set $K_n=\{x\in X\,|\, u_n(x)\leq -s_n\}$ for a suitable constant $s_n.$ Exponential estimates on $\widetilde E_\lambda$ will prove that its image by $f^i,$ $0\leq i\leq n,$ cannot be concentrated near $\widetilde E_\lambda.$ Therefore, volume estimates will imply that $f^n(K_n)=\{x\in X\,|\, u(x)\leq -d^ns_n\}$ is large if Theorem \ref{th la version u_n du main} is false on $X.$ Hence, a good choice of $s_n,$ allowed by the gap between $\lambda_1$ and $\lambda,$ will give a contradiction.

We first fix some constants. In Corollary \ref{cor loja} the constant $b$ depends only on $X.$ Then, by replacing $f$ by $f^n$ and $\delta$ by $\delta'^n$ with $b\delta^n<\delta'^n<\lambda_1^n,$ we can assume that $b=1.$ Let $0<A\leq1,$ $N\geq1$ be the other constants of Corollary \ref{cor loja}. Fix $\lambda_2,\lambda_3>0$ such that
\begin{itemize}
 \item $\delta<\lambda_1<\lambda_2<\lambda_3<\lambda,$
\item and $q>q_0$ large enough such that $\lambda_1/d<(\lambda_2/d)^{1+\epsilon}$ where $\epsilon$ and $q_0$ are defined in Lemma \ref{le Ho2: avec valeur L^p}.
\end{itemize}
Multiplicities of $f_{|X}$ are controlled outside $\widetilde E_\lambda.$ By induction hypothesis, we have a control of $u_n$ on $E_\lambda.$ We want to extend it to $\widetilde E_\lambda.$ Let $E$ be an irreducible component of $E_\lambda.$ The restriction of $f$ to each component of $(f_{|X})^{-1}(E)$ is surjective onto $E.$ Therefore, we deduce from Lemma \ref{le remonté d'estimée par surjection} that there exists $q'\geq1$ such that
\begin{align*}
 \|v\circ f\|_{L^q((f_{|X})^{-1}(E))}\lesssim\|v\|_{L^{qq'}(E)},
\end{align*}
for all psh modulo $T$ function $v$ on $\mathbb P^k.$ Hence, by induction hypothesis, there is a constant $M>0$ such that $\|u_n\|_{L^q(\widetilde E_{\lambda})}\leq M(\lambda_1/d)^n$ for $n\geq1.$ The next step is to obtain exponential estimates in a neighborhood of $\widetilde E_\lambda.$
\begin{lemma}\label{le Ho4: dans vois avec valeur L^p centrale}
 There exist constants $c,\eta\geq 1$ and $n_0\geq 1$ such that if $n\geq n_0$ then for all $u\in\cali F_\lambda(C)$ we have
\begin{equation*}
 \int_{\widetilde E_{\lambda,t_n}}\exp(-(d/\lambda_2)^nu_n)\omega^{l}\leq c,
\end{equation*}
where $t_n=(\lambda_2/d)^{n\eta}.$
\end{lemma}
\begin{proof}
 Let $E$ be an irreducible component of $\widetilde E_{\lambda}$ of dimension $i.$ According to the choice of $q,$ we can find $\lambda'_2<\lambda_2$ such that $\lambda_1/d<(\lambda'_2/d)^{1+\epsilon}.$ Hence
\begin{equation*}
 \|u_n\|_{L^q(E)}(d/\lambda'_2)^{(1+\epsilon)n}\leq M(\lambda_1/d)^n(d/\lambda'_2)^{(1+\epsilon)n}\leq M.
\end{equation*}
and by Lemma \ref{le Ho2: avec valeur L^p} with $s=(d/\lambda'_2)^n$ we have
\begin{equation*}
\int_E\exp(-a'(d/\lambda_2')^nu_n)\omega^i\leq c',
\end{equation*}
for some constants $a',c'>0.$ Therefore, if we set $\rho_n=(\lambda_2/d)^n,$ the volume in $E$ of $F_n=\{x\in E\,|\, u_n(x)\leq -\rho_n\}$ is smaller than $c'\exp(-a'(\lambda_2/\lambda'_2)^n).$ In particular, $F_n$ contains no ball of radius $\rho_n^{2/\alpha}$ for $n$ large enough.

If $X$ is smooth then set $t_n=\rho_n^{1/\alpha}.$ As in Lemma \ref{le Ho3: geometrie des sous niveau}, for $n$ large enough, we can find a covering of $E_{t_n}$ by balls with center in $E$ and of radius $2t_n$ on which Lemma \ref{le Ho1: avec valeur centrale} holds. Hence, we get
\begin{equation*}
 \int_{E_{t_n}}\exp(-au_n/\rho_n)\omega^{l}\leq c,
\end{equation*}
for some $a,c>0.$ The same argument with $\lambda_2$ slightly smaller shows that we can choose $a=1.$ We conclude the proof by summing on all irreducible components of $\widetilde E_{\lambda}.$

When $X$ is singular, we consider a desingularization $\pi:\widehat X\to X.$ In order to establish the estimate near $E,$ we proceed inductively as follows. Assume that there exists a triplet $(A,a,\theta)$ with $a>0,$ $\theta\geq1$ and an analytic set $A\subset E$ such that 
$$\int_{E_{t^{\theta}}\setminus A_{t^{1/\theta}}}\exp(-au_n/\rho_n)\omega^l$$
is uniformly bounded in $n\geq 0$ for $t\leq \rho_n.$ Then, using the properties of the elements of $\cali B_\lambda$ and dynamical arguments, we claim that a similar estimate holds if we substitute $(A,a,\theta)$ by some $(A',a',\theta')$ with $\dim(A')<\dim(A).$ It will give the result for $\eta$ large enough after less than $l$ steps since $\dim(E)<l.$

More precisely, let $V$ be an irreducible component of $A$ with maximal dimension. We
distinguish two cases, according to whether $V$ is in $\cali B_\lambda$ or not. In the first case, we know that for all $p\geq1,$ $\|u_n\|_{L^p(V)}\lesssim (\lambda_1/d)^n.$ We set $\widehat V:=\pi^{-1}(V).$ We denote by $\widehat V_1$ the union of all components of $\widehat V$ which are mapped onto $V$ and by $\widehat V_2$ the union of the other components of $\widehat V.$ Therefore, Lemma \ref{le remonté d'estimée par surjection} implies that
$$\|\widehat u_n\|_{L^p(\widehat V_1)}\lesssim (\lambda_1/d)^n,$$
for all $p\geq1,$ where $\widehat u_n=u_n\circ\pi.$ Hence, the smooth version of the lemma implies that
$$\int_{\widehat V_{1,\rho_n^{1/\alpha}}}\exp(-a'\widehat u_n/\rho_n)\pi^*(\omega^l)$$
is uniformly bounded for $a'>0$ small enough. Moreover, by Lemma \ref{le estime voisinage de transforme totale}, there exists a constant $\theta'\geq1$ such that $\pi(\widehat V_{1,t})$ contains $V_{t^{\theta'}}\setminus V_{2,t^{1/{2}}},$ where $V_2=\pi(\widehat V_2).$ It gives the desired result near $V,$ since $\dim(V_2)<\dim(V).$

From now, we can assume that no irreducible component of $A$ with maximal dimension belong to $\cali B_\lambda$ (in particular $A\neq E$). Let $V$ denote the union of all irreducible components of $A$ with maximal dimension. In particular, these components are not totally invariant for $f_{|E},$ therefore there exist an analytic set $Z\subset E$ containing no component of $V$ and an integer $m\geq1$ such that $f^m(Z)=V.$ We set $Z'=Z\cap A.$ 
The assumption on $A$ and $\theta$ implies that if $t\leq\rho_n$ then
$$\int_{Z_{t^\theta}\setminus A_{t^{1/\theta}}}\exp(-au_n/\rho_n)\omega^l$$
is bounded uniformly on $n.$ By Corollary \ref{cor iné de Loja classique pour ens}, $Z_{t_\theta}\cap A_{t^{1/\theta}}$ is contained in $Z'_{t^{1/\theta''}}$ for some $\theta''>\theta.$ So,
$$\int_{Z_{t^{\theta''}}\setminus Z'_{t^{1/\theta''}}}\exp(-au_n/\rho_n)\omega^l$$
is bounded uniformly on $n.$ Fix a constant $B>1$ large enough. We deduce from Corollary \ref{cor loja} applied to $\mathbb P^k$ that for all $t>0$ $f^{m}(Z_{t})$ contains $V_{B^{-1}t^{d^{mk}}}.$ Moreover, since $f^m$ is Lipschitz, $f^m(Z'_t)$ is contained in $V'_{Bt},$ where $V'=f^m(Z').$ So, we have
$$f^m(Z_{t^{\theta''}}\setminus Z'_{t^{1/\theta''}})\supset V_{t^{\theta'}}\setminus V'_{t^{1/\theta'}},$$
for $t>0$ small enough and $\theta'>\theta''$ large enough. It follows that
\begin{align}\label{eq dans Ho4}
\int_{V_{t^{\theta'}}\setminus V'_{t^{1/\theta'}}}\exp(-a'u_n/\rho_n)\omega^l&\leq\int_{Z_{t^{\theta''}}\setminus Z'_{t^{1/\theta''}}}\exp(-a'\frac{u_{n+m}\lambda_2^m}{\rho_{n+m}})(f_{|X}^m)^*(\omega^l)\nonumber\\
&\lesssim \int_{Z_{t^{\theta''}}\setminus Z'_{t^{1/\theta''}}}\exp(-a'\frac{u_{n+m}\lambda_2^m}{\rho_{n+m}})\omega^l,
\end{align}
since $(f_{|X}^m)^*(\omega^l)\lesssim\omega^l.$ Moreover, for $a'$ same enough the right-hand side in \eqref{eq dans Ho4} is bounded uniformly on $n$ and $\dim(V')=\dim(Z')<\dim(V)$ since $Z$ contains no component of $V.$ This together with the estimate outside $A$ prove the claim with $A'=V'.$
\end{proof}

From now, we fix $p\geq 1$ and for $u$ in $\cali F_\lambda(C)$ denote by $\cali N(u)=\{n\geq 1\,|\, \|u_n\|_{L^p(X)}\geq (\lambda/d)^n \}$ and by $\cali N$ the union of $\cali N(u)$ for all $u.$ Our goal is to prove that $\cali N$ is finite, which will imply Theorem \ref{th la version u_n du main}. For this purpose, we have the following result.
\begin{lemma}\label{le existence de grosses boules dans K_n}
There are constants $n_1\geq 1$ and $\beta\geq1$ such that if $n$ is in $\cali N(u)$ with $n\geq n_1$ then $K_n=\{x\in X\ |\ u_n(x)\leq -(\lambda_3/d)^n\}$ contains a ball of radius $(\lambda_3/d)^{\beta n}.$
\end{lemma}
\begin{proof}
Since $x^p\lesssim\exp(x)$ if $x\geq0,$ we deduce from the assumption on $\|u_n\|_{L^p(X)}$ that
\begin{align}\label{eq le existence de grosses boules dans K_n}
 (\lambda/\lambda_3)^n&\lesssim \left(\int_{X}(-(d/\lambda_3)^nu_n)^p\omega^{l}\right)^{1/p} \nonumber \\
&\lesssim\left(\int_{X}\exp(-(d/\lambda_3)^nu_n/2)\omega^{l}\right)^{1/p}.
\end{align}
On the other hand, let $\beta$ be the constant in Lemma \ref{le Ho3: geometrie des sous niveau}. For $n$ sufficiently large we have $(d/\lambda_3)^n\geq2.$ Hence, Lemma \ref{le Ho3: geometrie des sous niveau}  with $s=(d/\lambda_3)^n$ imply that $K_n$ has to contain a ball of radius $(\lambda_3/d)^{n\beta},$ otherwise the right-hand side of \eqref{eq le existence de grosses boules dans K_n} would be bounded uniformly on $n,$ which is a impossible since $\lambda_3<\lambda.$
\end{proof}
We can now complete the proof of the main theorem.
\begin{proof}[End of the proof of Theorem \ref{th la version u_n du main}]
 If $B\subset X$ is a Borel set then $|B|$ denotes its volume with respect to the measure $\omega^l.$ As we have already seen, the volume of a ball of radius $r$ in $X$ is larger than $c'r^{2l},$ $0<c'\leq1.$ Therefore, observe that if $x$ is in $\widetilde E_{\lambda,t_n/2}$ then $|\widetilde E_{\lambda,t_n}\cap B_{X}(x,r)|=|B_{X}(x,r)|\geq c'(r/2)^{2l}$ for $r<t_n/2.$

From now, assume in order to obtain a contradiction that $\cali N$ is infinite. Consider $u\in\cali F_\lambda(C)$ and $n\in\cali N(u)$ large enough.
Fix also $\beta$ large enough. So, we have $(\lambda_3/d)^{\beta n}<t_n/4$ and 
\begin{equation}\label{eq n assez grand pour avoir de bon rayon}
 c\exp(-(\lambda_3/\lambda_2)^n)\leq c'(A^\delta (t_n/2)^{N\delta}(\lambda_3/d)^{\beta n}/2)^{2l},
\end{equation}
where $c$ is defined in Lemma \ref{le Ho4: dans vois avec valeur L^p centrale}. Let $r_0=(\lambda_3/d)^{\beta n},$ and for $1\leq i\leq n$ let $r_i=A(t_n/2)^{N}r_{i-1}^\delta.$ We will prove by induction that for $0\leq i\leq n,$ $f^i(K_n)=\{x\in X\,|\, u_{n-i}(x)\leq -\lambda_3^n/d^{n-i}\}$ contains a ball $B_i$ of radius $r_i.$

Since $\beta$ is large, Lemma \ref{le existence de grosses boules dans K_n} implies that the assertion is true for $i=0.$ Let $0\leq i\leq n-1$ and assume the property is true for $i.$ We deduce from Lemma \ref{le Ho4: dans vois avec valeur L^p centrale} that
\begin{equation*}
 \int_{\widetilde E_{\lambda,t_{n-i}}}\exp(-(d/\lambda_2)^{n-i}u_{n-i})\omega^{l}\leq c,
\end{equation*}
and in particular
\begin{equation*}
 |\widetilde E_{\lambda,t_n}\cap B_i|\leq |\widetilde E_{\lambda,t_{n-i}}\cap f^i(K_n)|< c\exp(-(\lambda_3/\lambda_2)^n\lambda_2^i),
\end{equation*}
since $t_n\leq t_{n-i}.$ This and \eqref{eq n assez grand pour avoir de bon rayon} imply that
$$|B_i|\geq c'r_i^{2l}>2^{2l}|\widetilde E_{\lambda,t_n}\cap B_i|,$$
since $r_i\geq (A^\delta (t_n/2)^{N\delta}r_0)^{\delta^i}$ and $\delta<\lambda_2.$ Consequently, the center of $B_i$ is not in $\widetilde E_{\lambda,t_n/2}$ and by Corollary \ref{cor loja}, $f(B_i)\subset f^{i+1}(K_n)$ contains a ball $B_{i+1}$ of radius $r_{i+1}=A(t_n/2)^{N}r_i^\delta.$ Note that we already reduced the problem to the case where the constant $b$ in Corollary \ref{cor loja} is equal to $1.$ 

Therefore, for all $n$ in $\cali N(u)$ sufficiently large, the volume of $f^n(K_n)=\{x\in X\,|\, u(x)\leq -\lambda_3^n\}$ is greater than $D^{n\delta^n},$ with $0<D<1$ independent of $u$ and $n.$ This contradicts the inequality $\delta<\lambda_3.$ Indeed, since $\cali F_\lambda(C)$ is bounded in $L^q(X),$ by Lemma \ref{le Ho2: avec valeur L^p} there exists $a'>0$ such that
$$\int_{X}\exp(-a'u)\omega^{l}$$
is uniformly bounded for $u$ in $\cali F_\lambda(C)$

Hence, $\cali N$ is finite and in particular bounded by some $n_2\geq 1.$ We conclude using the fact that the restriction of $\cup_{n=0}^{n_2}d^{-n}(f^n)^*(\cali F_\lambda(C))$ to $X$ is a relatively compact family of wpsh modulo $T$ functions and then bounded in $L^p(X).$ Therefore, we have
\begin{equation*}
 \|u_n\|_{L^p(X)}\lesssim\left(\frac{\lambda}{d}\right)^n,
\end{equation*}
if $n\leq n_2$ and thus for every $n\geq 0$ by the definition of $\cali N.$
\end{proof}

\noindent
J. Taflin, UPMC Univ Paris 06, UMR 7586, Institut de
Math{\'e}matiques de Jussieu, F-75005 Paris, France. {\tt  taflin@math.jussieu.fr} 
\end{document}